\documentclass[11pt]{article}

\usepackage[T1]{fontenc}
\usepackage[utf8]{inputenc}
\usepackage{lmodern}
\usepackage{amsmath,amssymb,amsthm}
\usepackage{booktabs}
\usepackage{geometry}
\usepackage{hyperref}
\usepackage{array}
\usepackage{enumitem}

\hypersetup{
  hidelinks,
  pdftitle={A Mixed-Resonance Counterexample to the Lukic Conjecture},
  pdfauthor={Jun Yan}
}

\newtheorem{theorem}{Theorem}[section]
\newtheorem{proposition}[theorem]{Proposition}
\newtheorem{lemma}[theorem]{Lemma}
\newtheorem{remark}[theorem]{Remark}

\newcommand{\ZH}{Z_H}
\newcommand{\cE}{\mathcal E}
\newcommand{\cF}{\mathcal F}
\newcommand{\cL}{\mathcal L}

\title{A Mixed-Resonance Counterexample to the Lukic Conjecture}
\author{Jun Yan}
\date{25 July 2026}

\begin{document}
\maketitle

\begin{abstract}
Let $\mu$ be a probability measure on the unit circle with Verblunsky
coefficients $\alpha$. Lukic conjectured that a weighted entropy condition
with finitely many critical points is equivalent to a decomposition of
$\alpha$ into components localized at those points. We give a counterexample 
for two critical points of multiplicity three, found by GPT-5.6. The
construction uses two phase modes with common power-decay exponent
$3/20$. The sequence admits the Lukic's decomposition conditions, but its
weighted entropy for the corresponding two-point weight equals $-\infty$.

\end{abstract}

\section{Introduction}

Let
\[
  d\mu(\theta)=w(\theta)\frac{d\theta}{2\pi}+d\mu_s
\]
be a probability measure on the unit circle with infinite support, and let
$\alpha=(\alpha_n)$ be its Verblunsky coefficients. For distinct angles
$\theta_1,\dots,\theta_K$ and positive integers $m_1,\dots,m_K$, write
\begin{equation}
  H(e^{i\theta})
  =\prod_{j=1}^{K}\bigl(1-\cos(\theta-\theta_j)\bigr)^{m_j}.
  \label{eq:H-general}
\end{equation}
Lukic's modified conjecture asserts that the weighted entropy condition
\begin{equation}
  \int_0^{2\pi}H(e^{i\theta})\log w(\theta)
  \,\frac{d\theta}{2\pi}>-\infty
  \label{eq:entropy}
\end{equation}
is equivalent to the existence of sequences
$\beta^{(1)},\dots,\beta^{(K)}$ satisfying
\begin{equation}
  \alpha=\sum_{j=1}^{K}\beta^{(j)},\qquad
  (S-e^{-i\theta_j})^{m_j}\beta^{(j)}\in\ell^2,\qquad
  \beta^{(j)}\in\ell^{\,2m_j+2}.
  \label{eq:decomposition}
\end{equation}

\paragraph{Background and related work.}
The starting point is Szeg\H{o}'s classical theorem \cite{Szego}; standard
background on OPUC and Verblunsky coefficients is given in Simon's
monograph \cite{SimonOPUC}, while the broader ``spectral gem'' viewpoint is
developed in \cite{SimonGems}. Denisov and Kupin treated polynomially
weighted Szeg\H{o} classes \cite{DenisovKupin}; Simon and Zlato\v{s}
settled important two-singularity cases \cite{SimonZlatos}; and Golinskii
and Zlato\v{s} proved the general polynomial-weight equivalence under the
additional hypothesis $\alpha\in\ell^4$ \cite{GZ}.

Simon formulated an early general higher-order conjecture. Lukic found a
counterexample to its unrestricted multipoint form and proposed the
localized decomposition conditions \eqref{eq:decomposition}
\cite{LukicConjecture}. He subsequently proved a single-critical-point
theorem under an additional first-difference assumption
\cite{LukicSingle}. On the probabilistic side, Gamboa, Nagel, and Rouault
developed sum rules from large deviations, first for scalar spectral
measures and then in related matrix-measure settings
\cite{GNRcircle,GNRmatrix,GNR}. Breuer, Simon, and Zeitouni gave a
pedagogical development of this method \cite{BSZPedagogical} and then
derived the abstract sum rule used here, verifying several cases of
Lukic's conjecture, including a restricted $(2,1)$ case \cite{BSZ}.

Yan introduced an algebraic model for the homogeneous coefficient-side
polynomials, whose representative has a Hall--Littlewood form connected
with the symmetric-function framework of Macdonald \cite{Macdonald,Yan}.
He recovered the Golinskii--Zlato\v{s} theorem and proved the $K=1$,
arbitrary-order sufficiency direction
\eqref{eq:decomposition}$\Rightarrow$\eqref{eq:entropy}
\cite[Theorem~4]{Yan}. The analytic interpolation step in that argument is
a discrete Gagliardo--Nirenberg inequality, rooted in the classical work
of Gagliardo and Nirenberg and discussed in standard analysis references
\cite{Gagliardo,Nirenberg,SimonAnalysis,Taylor}. These results left the
unrestricted higher-multiplicity multipoint problem unresolved.

\begin{theorem}[Main result]
\label{thm:main}
The implication from the decomposition conditions
\eqref{eq:decomposition} to the entropy condition \eqref{eq:entropy} is
false. More precisely, take
\[
  K=2,\qquad (\theta_1,\theta_2)=(0,\pi/2),\qquad
  (m_1,m_2)=(3,3).
\]
For every $0<c<1/2$, the Verblunsky sequence
\begin{equation}
  \alpha_n=c(n+1)^{-3/20}\bigl(1+(-i)^n\bigr)
  \label{eq:main-alpha}
\end{equation}
has a decomposition satisfying \eqref{eq:decomposition}, whereas its
associated measure obeys
\begin{equation}
  \int_0^{2\pi}(1-\cos\theta)^3(1-\sin\theta)^3
  \log w(\theta)\,\frac{d\theta}{2\pi}=-\infty.
  \label{eq:main-entropy}
\end{equation}
\end{theorem}

A corrected theorem will need additional mixed-frequency conditions. The
exact coefficient table suggests that these conditions should be organized
by resonant oscillatory multiplicities of the homogeneous sum-rule
polynomials, rather than solely by the individual critical orders.

\begin{remark}
The result does not claim that every two-point choice fails. Indeed, Breuer, Simon,
and Zeitouni proved the sufficiency direction for one simple and one double
singularity, written $(2,1)$ in \cite[Theorem~9.1]{BSZ} and equivalently
$(1,2)$ after relabeling, under their antipodal-angle and
real-Verblunsky hypotheses.
\end{remark}

\begin{remark}
The proof uses the published abstract gem of Breuer, Simon, and Zeitouni
\cite{BSZ} and Yan's algebraic representative \cite{Yan}. The decisive
coefficient is also computed directly from a finite unitary GGT matrix, so
the conclusion does not rest on the quotient-algebra calculation introduced in \cite{Yan} alone.
\end{remark}

The exponent $3/20$ is chose by two competing thresholds. Each localized
component must lie in $\ell^8$, which requires an exponent strictly larger
than $1/8$. The surviving sextic term must fail to be summable, which
requires an exponent at most $1/6$. The value $3/20$ lies strictly between
them.

\paragraph{Statement of AI use.}
This example was generated by \mbox{GPT-5.6}. The author prompted
\mbox{GPT-5.6} during the process to extend the previous proof of $K=1$ case to the general
case, and the model found this example.

\section{The BSZ sum rule and Yan's algebraic representative}

\subsection{The measure-side functional}

Expand
\[
  H(e^{i\theta})=\sum_{\ell=-d}^{d}h_\ell e^{i\ell\theta},
  \qquad d=\sum_jm_j,
\]
and set
\begin{equation}
  \ZH=h_0=\frac{1}{2\pi}\int_0^{2\pi}H(e^{i\theta})\,d\theta,
  \label{eq:ZH}
\end{equation}
\begin{equation}
  d\eta(\theta)=\frac{H(e^{i\theta})}{\ZH}\frac{d\theta}{2\pi}.
  \label{eq:eta}
\end{equation}
The zeros of $H$ are finite, so $\eta$ has full support and a density that
is strictly positive almost everywhere. The Fourier identity gives the
Laurent polynomial potential
\begin{equation}
  V(z)=-\frac{1}{\ZH}
  \sum_{\substack{\ell=-d\\ \ell\ne0}}^{d}
  \frac{h_\ell}{|\ell|}z^\ell.
  \label{eq:potential}
\end{equation}
If $H(\eta\mid\mu)$ denotes relative entropy with $\eta$ in the first
argument, then
\begin{equation}
  H(\eta\mid\mu)
  =C_H-\frac{1}{\ZH}\int_0^{2\pi}
  H(e^{i\theta})\log w(\theta)\frac{d\theta}{2\pi},
  \label{eq:relative-entropy}
\end{equation}
where
\[
  C_H=\frac{1}{\ZH}\int H\log(H/\ZH)\,\frac{d\theta}{2\pi}
\]
is finite. Thus finiteness of $H(\eta\mid\mu)$ is equivalent to
\eqref{eq:entropy}.

\subsection{The finite-volume and infinite-volume coefficient sides}

Let $U_N$ be a finite $N\times N$ unitary CMV or GGT matrix with
$\alpha_{N-1}$ on the unit circle. BSZ prove that there are
$N$-independent boundary polynomials $F_-$ and $F_+$ and a local polynomial
$G$ such that
\begin{equation}
  \operatorname{Tr}V(U_N)
  =F_-+F_+
  +\sum_{n=0}^{N-1-d}G(\alpha_n,\dots,\alpha_{n+d}).
  \label{eq:local-trace}
\end{equation}

\begin{theorem}[BSZ abstract gem]
\label{thm:BSZ}
Let
\begin{equation}
  \mathcal B_N(\alpha)
  =\sum_{n=0}^{N}
  \left[
    G(\alpha_n,\dots,\alpha_{n+d})
    -\log(1-|\alpha_n|^2)
  \right].
  \label{eq:BSZ-bulk}
\end{equation}
For every Verblunsky sequence $\alpha$ and its associated measure $\mu$,
the extended-real limit
\[
  \lim_{N\to\infty}\mathcal B_N(\alpha)\in\mathbb R\cup\{+\infty\}
\]
exists, after the telescoping normalization of $G$ used by BSZ, and
\begin{equation}
  H(\eta\mid\mu)<\infty
  \quad\Longleftrightarrow\quad
  \lim_{N\to\infty}\mathcal B_N(\alpha)<\infty.
  \label{eq:abstract-gem-criterion}
\end{equation}
\end{theorem}

This is the infinite-volume formulation of
\cite[Theorems~3.5--3.6]{BSZ}.  It is the form used in the proof below.
The relation with genuine finite-unitary CMV or GGT functionals is a
consequence of \eqref{eq:local-trace}; the precise bounded-error comparison
needed here is recorded in Proposition~\ref{prop:comparison}.

\subsection{Yan's homogeneous representative}

Write $G=\sum_{k=1}^{d}G_{2k}$, where $G_{2k}$ is homogeneous of degree
$2k$ in $\alpha$ and its conjugate. Yan uses formal variables $x_j$ and
$y_j$ to record shifts of the $j$th unbarred and conjugated factors. Thus
$x_j^p$ records $\alpha_{n+p}$, while $y_j^q$ records
$\overline{\alpha_{n+q}}$. With these shift symbols, set
\begin{equation}
\begin{split}
  a_{k,p}&=\prod_{s=p}^{k}y_s\prod_{s=p+1}^{k}x_s,\\
  b_{k,q}&=\prod_{s=1}^{q}x_sy_s,\qquad
  R_k=\prod_{s=1}^{k}x_sy_s.
\end{split}
\label{eq:abR}
\end{equation}
Set
\begin{equation}
  \cL_{2k}
  =
  \sum_{p,q=1}^{k}
  \frac{H(a_{k,p}b_{k,q})}
  {\displaystyle
  \prod_{s\ne p}(1-a_{k,s}/a_{k,p})
  \prod_{t\ne q}(b_{k,q}/b_{k,t}-1)}.
  \label{eq:Yan-Laurent}
\end{equation}
Although the individual summands are rational, their sum is a finite
Laurent polynomial.  Choose an integer $M_k\ge0$, large enough that
\[
  R_k^{M_k}\cL_{2k}
  \in\mathbb C[x_1,y_1,\dots,x_k,y_k].
\]
Such an $M_k$ exists because $\cL_{2k}$ has only finitely many Laurent
monomials.  We use the ordinary-polynomial representative
\begin{equation}
  \cF_{2k}^{(M_k)}
  =\frac{(-1)^{k+1}}{k\ZH}R_k^{M_k}\cL_{2k}.
  \label{eq:Yan-representative}
\end{equation}
Because $R_k=1$ in Yan's quotient ring, multiplying by
$R_k^{M_k}$ does not change the quotient class.  Moreover, any two
admissible choices of $M_k$ give ordinary-polynomial representatives whose
partial sums differ by a uniformly bounded telescoping term; this is
verified in Section~7.2.  The precise rule translating the formal shift
symbols back to sequences is
\begin{equation}
  \left[
  \phi_{2k}\left(\prod_{j=1}^{k}x_j^{p_j}y_j^{q_j}\right)
  \right]_n
  =\prod_{j=1}^{k}\alpha_{n+p_j}\overline{\alpha_{n+q_j}}.
  \label{eq:phi}
\end{equation}
The variables $x_j,y_j$ are therefore not additional Verblunsky
coefficients. They are bookkeeping devices for shifts. Later we evaluate
them at phase numbers because a pure phase $\lambda^n$ is an eigenvector of
the shift:
\[
S(\lambda^n)=\lambda\,\lambda^n.
\]
This point is explained before the coefficient calculation in
Section~4.3.

The $-1/k$ term in Yan's quotient representative is combined with the
Taylor expansion of $\log(1-|\alpha_n|^2)$.

\section{Construction of the counterexample}

\subsection{Choice of phases and decay}

Take $\theta_0=0$, $\theta_1=\pi/2$, and $m_0=m_1=3$. The critical shift
phases are $\lambda_0=1$ and $\lambda_1=-i$. Put
\begin{equation}
  r_n=(n+1)^{-3/20},\qquad
  \beta_n^{(0)}=cr_n,\qquad
  \beta_n^{(1)}=c(-i)^nr_n,
  \label{eq:components}
\end{equation}
\begin{equation}
  \alpha_n=\beta_n^{(0)}+\beta_n^{(1)}
  =cr_n\bigl(1+(-i)^n\bigr).
  \label{eq:alpha-decomp}
\end{equation}
Since $|\alpha_n|\le 2c<1$, Verblunsky's theorem gives a unique probability
measure $\mu$ with infinite support and these coefficients.

\subsection{Verification of Lukic's decomposition conditions}

\begin{proposition}[Localized regularity]
\label{prop:regularity}
The decomposition
$\alpha=\beta^{(0)}+\beta^{(1)}$ satisfies all conditions in
\eqref{eq:decomposition}.
\end{proposition}

\begin{proof}
Because $8(3/20)=6/5>1$, both components belong to $\ell^8$. For every
fixed $q\ge0$,
\[
  \Delta^qr_n=O(n^{-3/20-q}).
\]
Consequently,
\[
  (S-1)^3\beta^{(0)}=O(n^{-63/20}),\qquad
  (S+i)^3\beta^{(1)}=O(n^{-63/20}),
\]
up to unimodular constants. Both sequences are square summable.
\end{proof}

It is also useful to record the full quadratic filter estimate
\begin{equation}
  (S-1)^3(S+i)^3\alpha\in\ell^2.
  \label{eq:quadratic-filter}
\end{equation}
Indeed, polynomials in $S$ commute and act boundedly on $\ell^2$. Therefore
\[
\begin{split}
(S-1)^3(S+i)^3\alpha
={}&(S+i)^3\bigl((S-1)^3\beta^{(0)}\bigr)\\
&+(S-1)^3\bigl((S+i)^3\beta^{(1)}\bigr)\in\ell^2.
\end{split}
\]

\section{The exact mixed sextic coefficient}

\subsection{The weight and its Fourier data}

For the chosen angles,
\begin{equation}
  H(e^{i\theta})
  =(1-\cos\theta)^3(1-\sin\theta)^3
  =\frac{i(z-1)^6(z-i)^6}{64z^6},
  \qquad z=e^{i\theta}.
  \label{eq:H-factor}
\end{equation}
Only even powers of $\cos\theta$ and $\sin\theta$ survive integration, so
\begin{equation}
  \ZH=1+\frac32+\frac32+\frac98=\frac{41}{8}.
  \label{eq:ZH-value}
\end{equation}
Equivalently, if $H(z)=\sum_{\ell=-6}^{6}h_\ell z^\ell$, then
$\ZH=h_0$. This number is the common normalization in
\eqref{eq:Yan-representative}. It turns the unnormalized value $6$ into the
actual coefficient
\[
\frac{6}{3\ZH}=\frac{16}{41}.
\]

\subsection{Removal of the divided-difference singularities}

For $k$ variables $a=(a_1,\dots,a_k)$ and $b=(b_1,\dots,b_k)$, define
\begin{equation}
  A_\ell(a)=\sum_{p=1}^{k}
  \frac{a_p^\ell}{\prod_{s\ne p}(1-a_s/a_p)},\qquad
  B_\ell(b)=\sum_{q=1}^{k}
  \frac{b_q^\ell}{\prod_{t\ne q}(b_q/b_t-1)}.
  \label{eq:AB}
\end{equation}
Let $h_m(u)$ be the complete homogeneous symmetric polynomial of degree
$m$, and let $e_k(u)=u_1\cdots u_k$. Lagrange interpolation gives, for
$\ell\ge0$,
\begin{equation}
\begin{split}
  A_\ell(a)&=h_\ell(a),\qquad B_0(b)=(-1)^{k-1},\\
  B_\ell(b)&=0\quad(1\le\ell<k),\\
  B_\ell(b)&=e_k(b)h_{\ell-k}(b)\quad(\ell\ge k).
\end{split}
\label{eq:positive-divided}
\end{equation}
For $m\ge1$, the reciprocal-variable identities are
\begin{equation}
\begin{split}
  A_{-m}(a)&=0\quad(m<k),\\
  A_{-m}(a)&=(-1)^{k-1}e_k(a^{-1})h_{m-k}(a^{-1})
  \quad(m\ge k),\\
  B_{-m}(b)&=(-1)^{k-1}h_m(b^{-1}).
\end{split}
\label{eq:negative-divided}
\end{equation}
These polynomial identities remain valid when phases coalesce and remove
every apparent $0/0$ term in \eqref{eq:Yan-representative}.

\subsection{Phase-selection evaluation}

For a pure phase $u_n=\lambda^n$, shifting by $p$ gives
\[
u_{n+p}=\lambda^pu_n.
\]
Hence a shift symbol $x^p$ is evaluated at $x=\lambda$. For a conjugated
phase, the corresponding eigenvalue is the complex conjugate of $\lambda$,
as displayed in \eqref{eq:phase-selections}. In our
decomposition the two phase bases are $\lambda_0=1$ and $\lambda_1=-i$.
Thus, for phase-selection vectors
$\varepsilon,\delta\in\{0,1\}^k$, evaluate the same shift variables that
occur in \eqref{eq:abR}--\eqref{eq:phi} by
\begin{equation}
  x_j=\lambda_{\varepsilon_j}\in\{1,-i\},\qquad
  y_j=\overline{\lambda_{\delta_j}}\in\{1,i\}.
  \label{eq:phase-selections}
\end{equation}
Let $p=|\varepsilon|$ and $q=|\delta|$ be the numbers of oscillatory
selections on the unbarred and conjugated sides. Their $n$-dependent phase is
\[
(-i)^{pn}i^{qn}=i^{(q-p)n}.
\]
We call the term \emph{resonant} when this phase is identically $1$,
equivalently when $q-p$ is divisible by four. For $k\le3$, this forces
$p=q$. The shared value $p=q$ is the \emph{common oscillatory
multiplicity}. Nonresonant terms retain a nontrivial fourth-root phase and
converge by oscillatory cancellation. At every resonant point, $R_k=1$.
Consequently, the resonant values are independent of the clearing exponent
$M_k$ chosen in \eqref{eq:Yan-representative}.

\begin{lemma}[Coefficient table]
\label{lem:coefficient-table}
The unnormalized resonant values and normalized coefficients are
\[
\begin{array}{c|c|c|c}
\text{degree}&\text{oscillatory multiplicity}&\text{raw}&\text{normalized}\\
\hline
2&0\text{ or }1&0&0\\
4&0,1,\text{ or }2&0&0\\
6&0\text{ or }3&0&0\\
6&1&6&16/41\\
6&2&6&16/41
\end{array}
\]
In particular,
\begin{equation}
  C_{2,1}=C_{1,2}=\frac{16}{41},\qquad
  C_{3,0}=C_{0,3}=0.
  \label{eq:mixed-coefficients}
\end{equation}
\end{lemma}

\begin{proof}
Here ``raw'' means the double sum in \eqref{eq:Yan-Laurent} before
multiplication by the common prefactor $(-1)^{k+1}/(k\ZH)$;
``normalized'' means the coefficient after that multiplication. Expanding
$H(z)=\sum_\ell h_\ell z^\ell$, the raw sum becomes
\[
R_k^{M_k}\sum_{\ell=-6}^{6}h_\ell A_\ell(a)B_\ell(b),
\]
which remains valid when phases coalesce by
\eqref{eq:positive-divided}--\eqref{eq:negative-divided}.
At a resonant selection $R_k=1$, so the displayed factor is $1$ and the
following values do not depend on $M_k$.

For $k=1$, the raw expression is
\[
H(x_1y_1^2)R_1^{M_1}.
\]
The two resonant substitutions are $(x_1,y_1)=(1,1)$ and $(-i,i)$,
and both have $R_1=1$; they therefore give $H(1)=0$ and $H(i)=0$.
Hence the two degree-two entries vanish.

For $k=2$, the four arguments appearing in the numerator are especially
simple:
\[
 a_{2,1}=y_1y_2x_2,\qquad a_{2,2}=y_2,\qquad
 b_{2,1}=x_1y_1,\qquad b_{2,2}=x_1y_1x_2y_2.
\]
For example, the common-oscillatory-multiplicity-one selection
\[
 (x_1,x_2)=(1,-i),\qquad (y_1,y_2)=(1,i)
\]
gives
\[
 \bigl(a_{2,p}b_{2,q}\bigr)_{p,q=1}^{2}
 =
 \begin{pmatrix}
  1&1\\
  i&i
 \end{pmatrix}.
\]
The other three multiplicity-one selections give the same two possible
arguments, $1$ and $i$. At multiplicity zero all four arguments equal
$1$, whereas at multiplicity two all four equal $i$.

It remains to justify the derivative order needed when two variables
coalesce. For $k=2$, each denominator in the $p$-sum has only one factor,
and
\[
 1-\frac{a_2}{a_1}=\frac{a_1-a_2}{a_1}.
\]
Thus the $p$-sum is a first divided difference in $a$; independently, the
$q$-sum is a first divided difference in $b$. If
$F(a,b)=H(ab)$, the two confluent limits are therefore governed by
\[
 \partial_aF(a,b)=bH'(ab),\qquad
 \partial_b\partial_aF(a,b)=H'(ab)+abH''(ab).
\]
Consequently the removable limits use only $H,H'$, and $H''$, never a
higher derivative. The factorization \eqref{eq:H-factor} gives
\[
H(1)=H'(1)=H''(1)=0,\qquad
H(i)=H'(i)=H''(i)=0,
\]
so every degree-four resonant substitution vanishes. This is not a
cancellation between nonzero phase selections: each selection pair is
zero.

For $k=3$, a binary word records which of the three factors uses the
oscillatory phase: $0$ means the phase $1$, while $1$ means $-i$ on the
unbarred side and $i$ on the conjugated side. Thus multiplicity one has
the three words
\[
 001,\qquad 010,\qquad 100.
\]
Choosing one word independently on each side gives $3\cdot3=9$ ordered
pairs. Evaluating the raw expression above, with rows and columns in that
order, gives
\[
 \begin{pmatrix}
  0&0&2\\
  2&0&0\\
  0&2&0
 \end{pmatrix},
 \qquad
 0+0+2+2+0+0+0+2+0=6.
\]
Hence ``the nine pairs of permutations of $001$ have raw sum $6$'' means
the sum of these nine displayed entries; only three pairs contribute, and
each contributes $2$.

For multiplicity two the words, in the order $011,101,110$, give instead
\[
 \begin{pmatrix}
  0&2&0\\
  0&0&2\\
  2&0&0
 \end{pmatrix},
 \qquad
 0+2+0+0+0+2+2+0+0=6.
\]
The pure multiplicities $0$ and $3$ have all phase choices equal. Their
second divided differences on the two sides require derivatives of $H$
only through total order four; these vanish because
\eqref{eq:H-factor} has a sixth-order zero at $1$ and at $i$. Thus the
pure raw sums are zero, while the two mixed raw sums are $6$. Finally,
$(-1)^{k+1}=1$ and $\ZH=41/8$, so each mixed normalized coefficient is
$6/[3(41/8)]=16/41$.
\end{proof}

\section{Passage to the slowly varying sequence}

\subsection{A frozen-amplitude expansion}

For each fixed integer $s$,
\begin{equation}
  r_{n+s}=r_n\left(1+O\left(\frac1n\right)\right).
  \label{eq:slow-variation}
\end{equation}

\begin{lemma}[Phase-selection expansion]
\label{lem:phase-selection-expansion}
Let $F$ be any finite-range homogeneous polynomial of degree $2k$ in
shifted copies of $\alpha$ and its conjugate. For the sequence
\eqref{eq:alpha-decomp}, $F_n$ is a finite sum of terms
\begin{equation}
  c^{2k}\omega^nr_n^{2k}
  +O\left(\frac{r_n^{2k}}{n}\right),
  \qquad \omega\in\{1,i,-1,-i\}.
  \label{eq:phase-selection-asymptotic}
\end{equation}
\end{lemma}

\begin{proof}
Expand every $\alpha$ factor into its constant-phase and oscillatory-phase
components and use
\eqref{eq:slow-variation} for each shifted amplitude. There are only
finitely many monomials and phase selections, so the errors remain
uniform.
\end{proof}

\subsection{Convergence part}

At degree two, the local polynomial is a positive multiple of the
quadratic filtered expression, hence its partial sums are bounded by
\eqref{eq:quadratic-filter}. At degree four, Lemma
\ref{lem:coefficient-table} shows that every resonant coefficient vanishes.
The remaining fourth-degree terms oscillate and converge by Dirichlet's
test. At degree six,
\begin{equation}
  [\phi_6(\cF_6^{(M_3)})]_n
  =\frac{32c^6}{41}(n+1)^{-9/10}
  +\text{oscillatory terms}
  +O(n^{-19/10}).
  \label{eq:sextic-local}
\end{equation}
For every $k\ge4$, the absolute value of the degree-$2k$ local term is
$O(r_n^{2k})$, and
\[
  2k(3/20)\ge6/5>1.
\]
Thus all higher-degree terms are absolutely summable. The logarithmic
remainder
\begin{equation}
  -\log(1-|\alpha_n|^2)
  -\sum_{k=1}^{6}\frac{|\alpha_n|^{2k}}{k}
  \label{eq:log-remainder}
\end{equation}
is nonnegative and $O(|\alpha_n|^{14})$, hence summable.

\subsection{Divergence and proof of the main theorem}

Define the Yan-representative partial sum by
\begin{equation}
\begin{split}
  S_N
  :=\sum_{n=0}^{N}\Bigg[
    &\sum_{k=1}^{6}[\phi_{2k}(\cF_{2k}^{(M_k)})]_n
    -\log(1-|\alpha_n|^2)\\
    &-\sum_{k=1}^{6}\frac{|\alpha_n|^{2k}}{k}
  \Bigg].
\end{split}
\label{eq:SN-definition}
\end{equation}
Replacing the shifted version of the term $-1/k$ in Yan's representative
by the unshifted last sum in \eqref{eq:SN-definition} changes partial sums
only by endpoint terms, hence by $O(1)$. Combining the preceding estimates
yields
\begin{equation}
  S_N=\frac{32c^6}{41}\sum_{n<N}(n+1)^{-9/10}+O(1)
  =\frac{320c^6}{41}N^{1/10}+O(1).
  \label{eq:divergence}
\end{equation}
Therefore $S_N\to+\infty$. Proposition~\ref{prop:comparison} below gives
\[
  \mathcal B_N(\alpha)=S_N+O(1),
\]
uniformly in $N$. Hence the BSZ limit is $+\infty$.
Theorem~\ref{thm:BSZ} and \eqref{eq:relative-entropy} then imply
\eqref{eq:main-entropy}, proving Theorem~\ref{thm:main}.

\section{Independent finite-unitary verification}

Sections 2--5 already prove the theorem, assuming Yan's published
homogeneous representative \eqref{eq:Yan-representative}. This section is
an independent consistency check, not an additional hypothesis. It
recomputes the decisive combined coefficient $32/41$ directly from the
finite-unitary BSZ functional, thereby checking the normalization,
conjugations, removable divided differences, and boundary convention.

\subsection{The genuine BSZ boundary convention}

Freeze the amplitude and set
\[
  \alpha_n^{(t)}=t\bigl(1+(-i)^n\bigr),\qquad 0\le n\le N-2,
  \qquad \alpha_{N-1}=1.
\]
Define
\[
  \rho_j=\bigl(1-|\alpha_j|^2\bigr)^{1/2}.
\]
The finite unitary GGT matrix has entries
\begin{equation}
  (U_N)_{k\ell}=
  -\alpha_{k-1}\overline{\alpha_\ell}
  \prod_{j=k}^{\ell-1}\rho_j\quad(k\le\ell),\qquad
  (U_N)_{\ell+1,\ell}=\rho_\ell,
  \label{eq:GGT}
\end{equation}
with $(U_N)_{k\ell}=0$ for $k\ge\ell+2$ and $\alpha_{-1}=-1$.
Define
\begin{equation}
  \cE_N(t)=\operatorname{Tr}V(U_N(t))
  -\sum_{n=0}^{N-2}\log(1-|\alpha_n^{(t)}|^2).
  \label{eq:finite-energy}
\end{equation}

\subsection{Exact Taylor coefficients}

The four periodic squared amplitudes are $4,2,0,2$ times $t^2$. For
$q\in\{0,2,4\}$,
\begin{equation}
  \sqrt{1-qt^2}
  =1-\frac q2t^2-\frac{q^2}{8}t^4-\frac{q^3}{16}t^6+O(t^8).
  \label{eq:rho-series}
\end{equation}
Substitution into \eqref{eq:GGT}, exact multiplication through degree six,
and the Fourier coefficients of Appendix~\ref{app:fourier} give
\begin{proposition}[Exact finite-unitary block check]
\label{prop:GGT-check}
With $[t^m]$ denoting Taylor coefficient extraction,
\begin{equation}
  \frac14[t^2](\cE_{12}-\cE_8)=0,\qquad
  \frac14[t^4](\cE_{12}-\cE_8)=0,\qquad
  \frac14[t^6](\cE_{12}-\cE_8)=\frac{32}{41}.
  \label{eq:GGT-check}
\end{equation}
\end{proposition}
Thus the genuine BSZ bulk polynomial has the same vanishing quadratic and
quartic frozen coefficients and the same positive sextic coefficient as
Yan's representative.

\section{Boundary, sign, and convergence check}

The local calculation in Lemma~\ref{lem:coefficient-table} identifies a
positive term, but a complete proof must also show that every omitted
contribution is bounded and cannot cancel its $N^{1/10}$ growth. This
section supplies those estimates. Its conclusions are used in the
divergence argument; they are separated here to keep the main coefficient
calculation readable.

\subsection{Finite-volume boundaries}

For comparison with the independent check in Section~6, define the
genuine finite-unitary functional
\begin{equation}
  \mathcal K_N(\alpha)
  =\operatorname{Tr}V(U_N)
  -\sum_{n=0}^{N-2}\log(1-|\alpha_n|^2),
  \label{eq:finite-unitary-functional}
\end{equation}
where $\alpha_{N-1}\in\partial\mathbb D$ is the terminal coefficient.
The BSZ boundary terms $F_-$ and $F_+$ are fixed polynomials in finitely
many coefficients. Since $|\alpha_n|\le1$, they are uniformly bounded.
For the slowly varying sequence, the right boundary has a finite limit
after fixing the terminal unit-circle parameter. It cannot cancel the
positive $N^{1/10}$ term in \eqref{eq:divergence}.
The difference between the full logarithmic sum in
\eqref{eq:finite-unitary-functional} and the logarithms paired with the
bulk range in
\eqref{eq:local-trace} contains only the fixed number $d-1$ of terms
with indices $N-d,\dots,N-2$.  Since $\alpha_n\to0$, this difference
tends to zero. Consequently,
\[
  \mathcal K_N(\alpha)
  =\mathcal B_{N-1-d}(\alpha)+O(1).
\]
This finite-volume comparison is not being substituted for the published
BSZ abstract gem; it is a consequence of the trace decomposition
\eqref{eq:local-trace}.

\subsection{Quotient representatives}

Let $R=\prod x_jy_j$. If two ordinary polynomial representatives differ
by $(R-1)Q$, and $q_n=[\phi(Q)]_n$, then
\begin{equation}
  \sum_{n=0}^{N}[\phi((R-1)Q)]_n
  =\sum_{n=0}^{N}(q_{n+1}-q_n)=q_{N+1}-q_0.
  \label{eq:telescoping}
\end{equation}
Every monomial in $q_n$ is bounded by one, so all quotient changes
contribute $O(1)$ uniformly in $N$.

In particular, let $M_k'>M_k$ be two exponents that both clear the
negative powers of $\cL_{2k}$.  (The case $M_k'=M_k$ is trivial.) Writing
$P=R_k^{M_k}\cL_{2k}$, which is an ordinary polynomial, gives
\[
 R_k^{M_k'}\cL_{2k}-R_k^{M_k}\cL_{2k}
 =(R_k-1)
 \bigl(1+R_k+\cdots+R_k^{M_k'-M_k-1}\bigr)P.
\]
Thus different admissible clearing exponents change the coefficient-side
partial sums by only $O(1)$.  This also proves the assertion following
\eqref{eq:Yan-representative}.

\begin{proposition}[Uniform comparison of coefficient sides]
\label{prop:comparison}
For the sequence \eqref{eq:main-alpha}, the BSZ bulk partial sums, Yan
partial sums, and genuine finite-unitary functionals satisfy
\begin{equation}
  \mathcal B_N(\alpha)=S_N+O(1),\qquad
  \mathcal K_N(\alpha)
  =\mathcal B_{N-1-d}(\alpha)+O(1),
  \label{eq:three-way-comparison}
\end{equation}
where both error bounds are uniform in $N$.
\end{proposition}

\begin{proof}
Yan's theorem identifies, degree by degree, the quotient class of the
homogeneous part $G_{2k}$ with the class represented by
\eqref{eq:Yan-representative}, including the term $-1/k$.  The difference
of two ordinary representatives of the same quotient class is
$(R_k-1)Q_k$.  Summing its image telescopes by
\eqref{eq:telescoping}, and the endpoint values are uniformly bounded
because $|\alpha_n|\le1$.  Moving the common shift in the representative
of $-1/k$ back to $|\alpha_n|^{2k}/k$ is another finite endpoint
telescoping. Summing over the finitely many degrees $1\le k\le d=6$
proves the first identity in \eqref{eq:three-way-comparison}.

For the second identity, insert the BSZ trace decomposition
\eqref{eq:local-trace} into \eqref{eq:finite-unitary-functional}.
The two boundary polynomials are uniformly bounded, and only $d-1$
right-edge logarithms remain unmatched.  Those logarithms tend to zero
because $\alpha_n\to0$. This proves the second identity.
\end{proof}

\section{Conclusion}

The general Lukic conjecture fails because localized components can
interact through mixed resonances that are invisible to the separate
$\ell^{2m_j+2}$ conditions. For two third-order critical points, all frozen
resonances through degree four vanish, but the mixed sextic coefficient is
positive. The interval
\[
  \frac18<\frac3{20}<\frac16
\]
turns that algebraic obstruction into a divergent sum while preserving
every proposed decomposition condition.

A corrected theorem will need additional mixed-frequency conditions. The
exact coefficient table suggests that these conditions should be organized
by resonant oscillatory multiplicities of the homogeneous sum-rule polynomials,
rather than solely by the individual critical orders.

\appendix

\section{Exact Fourier and resonance data}
\label{app:fourier}

The nonzero Laurent coefficients of \eqref{eq:H-factor} for nonnegative
indices are
\[
\begin{array}{c|c}
\ell&h_\ell\\
\hline
0&41/8\\
1&-51/16+(51/16)i\\
2&-(195/64)i\\
3&35/32+(35/32)i\\
4&-9/16\\
5&3/32-(3/32)i\\
6&(1/64)i
\end{array}
\]
and $h_{-\ell}=\overline{h_\ell}$. The complete resonant raw sums obtained
from \eqref{eq:AB}--\eqref{eq:negative-divided} are
\begin{equation}
  k=1:(0,0),\qquad
  k=2:(0,0,0),\qquad
  k=3:(0,6,6,0).
  \label{eq:raw-sums}
\end{equation}
The entries are ordered by the common oscillatory multiplicity. For
$k=3$, multiplication by $1/(3\ZH)$ converts the two middle entries to
$16/41$.


\begin{thebibliography}{99}
\footnotesize

\bibitem{BSZ}
J. Breuer, B. Simon, and O. Zeitouni,
\emph{Large deviations and the Lukic conjecture},
Duke Math. J. \textbf{167} (2018), no.~15, 2857--2902;
\href{https://arxiv.org/abs/1703.00653}{arXiv:1703.00653}.

\bibitem{BSZPedagogical}
J. Breuer, B. Simon, and O. Zeitouni,
\emph{Large deviations and sum rules for spectral theory---a pedagogical
approach},
J. Spectr. Theory, to appear (as cited in \cite{Yan}).

\bibitem{DenisovKupin}
S. Denisov and S. Kupin,
\emph{Asymptotics of the orthogonal polynomials for the Szeg\H{o} class
with a polynomial weight},
J. Approx. Theory (2006), 8--28.

\bibitem{Gagliardo}
E. Gagliardo,
\emph{Propriet\`a di alcune classi di funzioni in pi\`u variabili},
Ricerche Mat. (1958), 102--137.

\bibitem{GNRmatrix}
F. Gamboa, J. Nagel, and A. Rouault,
\emph{Sum rules and large deviations for spectral matrix measures},
preprint (as cited in \cite{Yan}).

\bibitem{GNRcircle}
F. Gamboa, J. Nagel, and A. Rouault,
\emph{Sum rules and large deviations for spectral measures on the unit
circle},
preprint (as cited in \cite{Yan}).

\bibitem{GNR}
F. Gamboa, J. Nagel, and A. Rouault,
\emph{Sum rules via large deviations},
J. Funct. Anal. \textbf{270} (2016), 509--559.

\bibitem{GZ}
L. Golinskii and A. Zlato\v{s},
\emph{Coefficients of orthogonal polynomials on the unit circle and
higher-order Szeg\H{o} theorems},
Constr. Approx. \textbf{26} (2007), 361--382.

\bibitem{LukicConjecture}
M. Lukic,
\emph{On a conjecture for higher-order Szeg\H{o} theorems},
Constr. Approx. \textbf{38} (2013), 161--169;
\href{https://arxiv.org/abs/1210.6953}{arXiv:1210.6953}.

\bibitem{LukicSingle}
M. Lukic,
\emph{On higher-order Szeg\H{o} theorems with a single critical point of
arbitrary order},
Constr. Approx., to appear (as cited in \cite{Yan}).

\bibitem{Macdonald}
I. G. Macdonald,
\emph{Symmetric Functions and Hall Polynomials},
2nd ed., Clarendon Press, 1998.

\bibitem{Nirenberg}
L. Nirenberg,
\emph{On elliptic partial differential equations},
Ann. Scuola Norm. Sup. Pisa (1959), 115--162.

\bibitem{SimonOPUC}
B. Simon,
\emph{Orthogonal Polynomials on the Unit Circle, Part 1: Classical Theory},
AMS Colloquium Publications 54.1, 2005.

\bibitem{SimonGems}
B. Simon,
\emph{Szeg\H{o}'s Theorem and Its Descendants: Spectral Theory for
$L^2$ Perturbations of Orthogonal Polynomials},
Princeton University Press, Princeton, NJ, 2011.

\bibitem{SimonAnalysis}
B. Simon,
\emph{A Comprehensive Course in Analysis, Part 3: Harmonic Analysis},
American Mathematical Society, Providence, RI, 2015.

\bibitem{SimonZlatos}
B. Simon and A. Zlato\v{s},
\emph{Higher-order Szeg\H{o} theorems with two singular points},
J. Approx. Theory \textbf{134} (2005), no.~1, 114--129;
\href{https://arxiv.org/abs/math-ph/0409065}{arXiv:math-ph/0409065}.

\bibitem{Szego}
G. Szeg\H{o},
\emph{Orthogonal Polynomials},
American Mathematical Society, 1939.

\bibitem{Taylor}
M. Taylor,
\emph{Partial Differential Equations III: Nonlinear Equations},
2nd ed., Springer, New York, 2011.

\bibitem{Yan}
J. Yan,
\emph{An algebra model for the higher-order sum rules},
Constr. Approx. \textbf{48} (2018), 453--471;
\href{https://arxiv.org/abs/1706.07925}{arXiv:1706.07925}.

\end{thebibliography}
\end{document}